\documentclass[reqno,11pt]{article}%
\usepackage{amsmath}
\usepackage{graphicx}
\usepackage{amsfonts}
\usepackage{amssymb}%
 \makeatletter
\newtheorem{theorem}{Theorem}

\newtheorem{proposition}[theorem]{Proposition}

\newenvironment{proof}[1][Proof]{\noindent{\textbf {#1}  }}  {\hfill$\Box$}

\begin{document}

\title{Some extremal problems for hereditary properties of graphs }
\author{Vladimir Nikiforov\thanks{Department of Mathematical Sciences, University of
Memphis, Memphis TN 38152, USA; email: \textit{vnikifrv@memphis.edu}}}
\maketitle

\begin{abstract}
Let $\mathcal{P}$ be an infinite hereditary property of graphs. Define%
\[
\pi\left(  \mathcal{P}\right)  =\lim_{n\rightarrow\infty}\binom{n}{2}^{-1}%
\max\{e\left(  G\right)  :\text{ }G\in\mathcal{P}\text{ and }v\left(
G\right)  =n\}.\text{ \ \ }%
\]
In this note $\pi\left(  \mathcal{P}\right)  $ is determined for every
hereditary property $\mathcal{P}$.

The same problem is studied for a more general parameter $\lambda^{\left(
\alpha\right)  }\left(  G\right)  ,$ defined for every real number $\alpha
\geq1$ and every graph $G\ $as
\[
\lambda^{\left(  \alpha\right)  }\left(  G\right)  =\max_{\left\vert
x_{1}\right\vert ^{\alpha}\text{ }+\text{ }\left\vert x_{2}\right\vert
^{\alpha}\text{ }+\text{ }\cdots\text{ }+\text{ }\left\vert x_{n}\right\vert
^{\alpha}\text{ }=\text{ }1}2\sum_{\{u,v\}\in E\left(  G\right)  }x_{u}x_{v}.
\]

It is known that the limit%
\[
\lambda^{\left(  \alpha\right)  }\left(  \mathcal{P}\right)  =\lim
_{n\rightarrow\infty}n^{2/\alpha-2}\max\{\lambda^{\left(  \alpha\right)
}\left(  G\right)  :\text{ }G\in\mathcal{P}\text{ and }v\left(  G\right)
=n\}\text{ \ \ }%
\]
exists. A key result of the note is the equality%
\[
\lambda^{(\alpha)}\left(  \mathcal{P}\right)  =\pi\left(  \mathcal{P}\right)
,
\]
which holds for all $\alpha>1.$

\end{abstract}

\section{Introduction}

In this note we study problems stemming from the following one:\ \medskip

\emph{What is the maximum number of edges a graph of order }$n,$
\emph{belonging to some} \emph{hereditary property }$\mathcal{P}$%
\emph{.}\medskip

Let us recall that a hereditary property is a family of graphs closed under
taking induced subgraphs. For example, given a set of graphs $\mathcal{F},$
the family of all graphs that do not contain any $F\in\mathcal{F}$ as an
induced subgraph is a hereditary property, denoted as $Her\left(
\mathcal{F}\right)  .$

It seems that the above classically shaped problem has been disregarded in the
rich literature on hereditary properties, so we fill in this gap below.

Writing $\mathcal{P}_{n}$ for the set of all graphs of order $n$ in a property
$\mathcal{P},$ our problem now reads as: \emph{Given a hereditary property
}$\mathcal{P}$\emph{, find}
\begin{equation}
ex\left(  \mathcal{P},n\right)  =\max_{G\in\mathcal{P}_{n}}e\left(  G\right)
. \label{maxed}%
\end{equation}

Finding $ex\left(  \mathcal{P},n\right)  $ exactly seems hopeless for
arbitrary $\mathcal{P}$. A more feasible approach has been suggested by
Katona, Nemetz and Simonovits in \cite{KNS64} who proved the following fact:

\begin{proposition}
\label{proKNS}If $\mathcal{P}$ is a hereditary property, then the sequence
\[
\left\{  ex\left(  \mathcal{P},n\right)  \binom{n}{2}^{-1}\right\}
_{n=1}^{\infty}%
\]
is nonincreasing and so the limit
\[
\pi\left(  \mathcal{P}\right)  =\lim_{n\rightarrow\infty}ex\left(
\mathcal{P},n\right)  \binom{n}{2}^{-1}%
\]
always exists.
\end{proposition}

One of the aims of this paper is to establish $\pi\left(  \mathcal{P}\right)
$ for every $\mathcal{P},$ but our main interest is in extremal problems about
a different graph parameter, denoted by $\lambda^{\left(  \alpha\right)
}\left(  G\right)  $ and defined as follows: \emph{for every graph }%
$G$\emph{\ and every real number }$\alpha\geq1,$\emph{ let}
\[
\lambda^{\left(  \alpha\right)  }\left(  G\right)  =\max_{\left\vert
x_{1}\right\vert ^{\alpha}+\cdots+\left\vert x_{n}\right\vert ^{\alpha}%
=1}2\sum_{\left\{  u,v\right\}  \in E\left(  G\right)  }x_{u}x_{v}.
\]

Note first that $\lambda^{\left(  2\right)  }\left(  G\right)  $ is the
well-studied spectral radius of $G,$ and second, that $\lambda^{\left(
1\right)  }\left(  G\right)  $ is a another much studied parameter, known as
the Lagrangian of $G$. So $\lambda^{\left(  \alpha\right)  }\left(  G\right)
$ is a common generalization of two parameters that have been widely used in
extremal graph theory.

The parameter $\lambda^{\left(  \alpha\right)  }\left(  G\right)  $ has been
recently introduced and studied for uniform hypergraphs first, by Keevash,
Lenz and Mubayi in \cite{KLM13} and next by the author, in \cite{NikB}. Here
we shall study $\lambda^{\left(  \alpha\right)  }\left(  G\right)  $ in the
same setting as the number of edges in (\ref{maxed}). Thus, given a hereditary
property $P$, set%
\begin{equation}
\lambda^{\left(  \alpha\right)  }\left(  \mathcal{P},n\right)  =\max
_{G\in\mathcal{P}_{n}}\lambda^{\left(  \alpha\right)  }\left(  G\right)  .
\label{maxlam}%
\end{equation}

As with $ex\left(  \mathcal{P},n\right)  $ finding $\lambda^{\left(
\alpha\right)  }\left(  \mathcal{P},n\right)  $ seems hopeless for arbitrary
$\mathcal{P}$. So, to begin with, the following theorem has been proved in
\cite{NikB} as an analog to Proposition \ref{proKNS}.

\begin{theorem}
\label{thlim}Let $\alpha\geq1.$ If $\mathcal{P}$ is a hereditary property,
then the limit
\begin{equation}
\lambda^{\left(  \alpha\right)  }\left(  \mathcal{P}\right)  =\lim
_{n\rightarrow\infty}\lambda^{\left(  \alpha\right)  }\left(  \mathcal{P}%
,n\right)  n^{\left(  2/\alpha\right)  -2}%
\end{equation}
exists.
\end{theorem}

Thus, a natural question is to find $\lambda^{\left(  \alpha\right)  }\left(
\mathcal{P}\right)  $ for every $\mathcal{P}$ and every $\alpha\geq1.$ The
main goal of this note to answer this question completely.

It turns out that $\lambda^{\left(  \alpha\right)  }\left(  \mathcal{P}%
\right)  $ and $\pi\left(  \mathcal{P}\right)  $ are closely related. For
example, results proved in \cite{NikB} imply that $\lambda^{\left(  a\right)
}\left(  \mathcal{P}\right)  \geq\pi\left(  \mathcal{P}\right)  $ for every
$\mathcal{P}$ and every $\alpha\geq1,$ moreover, if $\alpha\geq2,$ then
$\lambda^{\left(  a\right)  }\left(  \mathcal{P}\right)  =\pi\left(
\mathcal{P}\right)  .$ In this note we shall extend this relation to: \emph{if
}$\alpha>1,$\emph{ then }$\lambda^{\left(  a\right)  }\left(  \mathcal{P}%
\right)  =\pi\left(  \mathcal{P}\right)  .$

\section{Main results}

For notation and concepts undefined here, the reader is referred to
\cite{Bol98}.

Note first that every hereditary property $\mathcal{P}$ is trivially
characterized by $\mathcal{P}=Her\left(  \overline{\mathcal{P}}\right)  ,$
where $\overline{\mathcal{P}}$ is the family of all graphs that are not in
$\mathcal{P};$ however, typically $\mathcal{P}$ can be given as $\mathcal{P}%
=Her\left(  \mathcal{F}\right)  $ for some $\mathcal{F}$ that is only a small
fraction of $\overline{\mathcal{P}}.$

Recall next that a complete $r$-partite graph is a graph whose vertices are
split into $r$ nonempty independent sets so that all edges between vertices of
different classes are present. In particular, a $1$-partite graph is just a
set of independent vertices.

To characterize $\pi\left(  \mathcal{P}\right)  $ and $\lambda^{\left(
\alpha\right)  }\left(  \mathcal{P}\right)  $ we shall need two numeric
parameters defined for every family of graphs $\mathcal{F}.$ First, let
\[
\underline{\omega}\left(  \mathcal{F}\right)  =\left\{
\begin{array}
[c]{l}%
0,\text{ if }\mathcal{F}\text{ contains no cliques;}\\
\min\left\{  r:\text{ }K_{r}\in\mathcal{F}\right\}  ,\text{ otherwise,}%
\end{array}
\right.
\]
and second, let
\[
\beta\left(  \mathcal{F}\right)  =\left\{
\begin{array}
[c]{l}%
0,\text{ if }\mathcal{F}\text{ contains no complete partite graphs;}\\
\min\left\{  r:\mathcal{F}\text{ contains a complete }r\text{-partite
graph}\right\}  ,\text{ otherwise.}%
\end{array}
\right.
\]

The parameters $\underline{\omega}\left(  \mathcal{F}\right)  $ and
$\beta\left(  \mathcal{F}\right)  $ are quite informative about the hereditary
property $Her\left(  \mathcal{F}\right)  ,$ as seen first in the following observation.

\begin{proposition}
\label{proinf}If the property $\mathcal{P}$ $=Her\left(  \mathcal{F}\right)  $
is infinite, then $\underline{\omega}\left(  \mathcal{F}\right)  =0$ or
$\underline{\omega}\left(  \mathcal{F}\right)  \geq2$ and $\beta\left(
\mathcal{F}\right)  \geq2.$
\end{proposition}

\begin{proof}
Suppose that $\underline{\omega}\left(  \mathcal{F}\right)  \neq0.$ If
$\underline{\omega}\left(  \mathcal{F}\right)  =1,$ then $\mathcal{P}$ is
empty, so we can suppose that $\underline{\omega}\left(  \mathcal{F}\right)
\geq2.$ This implies that $\beta\left(  \mathcal{F}\right)  >0,$ as
$\mathcal{F}$ contains $K_{r}$ for some $r\geq2$ and $K_{r}$ is a complete
$r$-partite graph. If $\beta\left(  \mathcal{F}\right)  =1,$ then
$\mathcal{F}$ contains a graph $G$ consisting of isolated vertices, say $G$ is
on $s$ vertices. If $\mathcal{P}$ is infinite, choose a member $G\in
\mathcal{P}$ with $v\left(  G\right)  \geq r\left(  K_{r},K_{s}\right)  ,$
where $r\left(  K_{r},K_{s}\right)  $ is the Ramsey number of $K_{r}$ vs.
$K_{s}.$ Then either $G$ contains a $K_{r}$ or an independent set on $s$
vertices, both of which are forbidden. It turns out that $\beta\left(
\mathcal{F}\right)  \geq2,$ proving Proposition \ref{proinf}.
\end{proof}

\medskip

Clearly the study of (\ref{maxed}) and (\ref{maxlam}) makes sense only if
$\mathcal{P}$ is infinite and Proposition \ref{proinf} provides necessary
condition for this property of $\mathcal{P}$. The following theorem completely
characterizes $\pi\left(  \mathcal{P}\right)  .$

\begin{theorem}
\label{thpi}Let $\mathcal{F}$ be a family of graphs. If the property
$\mathcal{P}=Her\left(  \mathcal{F}\right)  $ is infinite, then
\[
\pi\left(  \mathcal{P}\right)  =\left\{
\begin{array}
[c]{l}%
1,\text{ if }\underline{\omega}\left(  \mathcal{F}\right)  =0;\\
1-\frac{1}{\beta\left(  \mathcal{F}\right)  -1},\text{ otherwise.}%
\end{array}
\right.  .
\]

\end{theorem}

\begin{proof}
Indeed, since $\mathcal{P}$ is infinite, Proposition \ref{proinf} implies that
$\underline{\omega}\left(  \mathcal{F}\right)  =0$ or $\underline{\omega
}\left(  \mathcal{F}\right)  \geq2$ and $\beta\left(  \mathcal{F}\right)
\geq2.$ If $\underline{\omega}\left(  \mathcal{F}\right)  =0,$ then $K_{n}%
\in\mathcal{P}_{n},$ because all subgraphs of $K_{n}$ are complete and do not
belong to $\mathcal{F}$. Therefore,
\[
ex\left(  \mathcal{P},n\right)  =\binom{n}{2},
\]
and so, $\pi\left(  \mathcal{P}\right)  =1.$ Assume that $\underline{\omega
}\left(  \mathcal{F}\right)  \geq2$ and $\beta\left(  \mathcal{F}\right)
\geq2,$ and set for short $r=\underline{\omega}\left(  \mathcal{F}\right)
\geq2$ and $\beta=\beta\left(  \mathcal{F}\right)  .$ Next, we shall prove
that $T_{\beta-1}\left(  n\right)  \in\mathcal{P}_{n},$ where $T_{\beta
-1}\left(  n\right)  $ is the complete $\left(  \beta-1\right)  $-partite
Tur\'{a}n graph of order $n.$ Indeed all subgraphs of $T_{\beta-1}\left(
n\right)  $ are complete $r$-partite graphs for some $r\leq\beta-1$, so should
one of them belong to $\mathcal{F},$ we would have $\beta\left(
\mathcal{F}\right)  \leq\beta-1=\beta\left(  \mathcal{F}\right)  -1,$ a
contradiction. Therefore,
\[
ex\left(  \mathcal{P},n\right)  \geq e\left(  T_{\beta-1}\left(  n\right)
\right)  =\left(  1-\frac{1}{\beta-1}+o\left(  1\right)  \right)  \binom{n}%
{2},
\]
and so%
\[
\pi\left(  \mathcal{P}\right)  \geq1-\frac{1}{\beta\left(  \mathcal{F}\right)
-1}.
\]

To finish the proof we shall prove the opposite inequality. Let $F\in
\mathcal{F}$ be a complete $\beta$-partite graph, known to exist by the
definition of $\beta\left(  \mathcal{F}\right)  $ and let $s$ be the maximum
of the sizes of its vertex classes.

Now assume that $\varepsilon>0$ and set $t=r\left(  K_{r},K_{s}\right)  ,$
where $r\left(  K_{r},K_{s}\right)  $ is the Ramsey number of $K_{r}$ vs.
$K_{s}.$ If $n$ is large enough and $G\in\mathcal{P}_{n}$ satisfies
\[
e\left(  G\right)  >\left(  1-\frac{1}{\beta\left(  \mathcal{F}\right)
-1}+\varepsilon\right)  \binom{n}{2},
\]
then by the theorem of Erd\H{o}s and Stone \cite{ErSt46}, $G$ contains a
subgraph $G_{0}=K_{\beta}\left(  t\right)  ,$ that is to say, a complete
$\beta$-partite graph with $t$ vertices in each vertex class. Since $K_{r}%
\in\mathcal{F}$, we see that $G_{0}$ contains no $K_{r},$ hence each vertex
class of $G_{0}$ contains an independent set of size $s,$ and so $G$ contains
an induced subgraph $K_{\beta}\left(  s\right)  ,$ which in turn contains an
induced copy of $F.$ Hence, if $n$ is large enough and $G\in\mathcal{P}_{n},$
then
\[
e\left(  G\right)  \binom{n}{2}^{-1}\leq1-\frac{1}{\beta\left(  \mathcal{F}%
\right)  -1}+\varepsilon.
\]
This inequality implies that
\[
\pi\left(  \mathcal{P}\right)  \leq1-\frac{1}{\beta\left(  \mathcal{F}\right)
-1},
\]
completing the proof.
\end{proof}

We continue now with establishing $\lambda^{\left(  \alpha\right)  }\left(
\mathcal{P}\right)  $ for $\alpha>1.$ The proof of our key Theorem \ref{thlam}
relies on several other results, some of which are stated within the proof
itself. We give two other before the theorem. The first one follows from a
result in \cite{NikB}, but for reader's sake we reproduce its short proof here.

\begin{theorem}
\label{maxlK}Let $\alpha\geq1.$ If $G$ is a graph with $m$ edges and $n$
vertices, with no $K_{r+1},$ then%
\begin{equation}
\lambda^{\left(  \alpha\right)  }\left(  G\right)  \leq\left(  1-\frac{1}%
{r}\right)  ^{1/\alpha}\left(  2m\right)  ^{1-1/\alpha} \label{in1}%
\end{equation}
and
\begin{equation}
\lambda^{\left(  \alpha\right)  }\left(  G\right)  \leq\left(  1-\frac{1}%
{r}\right)  n^{2-2/\alpha}. \label{in2}%
\end{equation}

\end{theorem}

\begin{proof}
Indeed, let $\mathbf{x}=\left(  x_{1},\ldots,x_{n}\right)  $ be a vector such
that $\left\vert x_{1}\right\vert ^{\alpha}+\cdots+\left\vert x_{n}\right\vert
^{\alpha}=1$ and
\[
\lambda^{\left(  \alpha\right)  }\left(  G\right)  =2\sum_{\left\{
u,v\right\}  \in E\left(  G\right)  }x_{u}x_{v}.
\]
Applying Jensen's inequality, we see that
\begin{align*}
\lambda^{\left(  \alpha\right)  }\left(  G\right)   &  =2\sum_{\left\{
u,v\right\}  \in E\left(  G\right)  }x_{u}x_{v}\leq2\sum_{\left\{
u,v\right\}  \in E\left(  G\right)  }\left\vert x_{u}\right\vert \left\vert
x_{v}\right\vert \\
&  \leq\left(  2m\right)  ^{1-1/\alpha}\left(  2\sum_{\left\{  u,v\right\}
\in E\left(  G\right)  }\left\vert x_{u}\right\vert ^{\alpha}\left\vert
x_{v}\right\vert ^{\alpha}\right)  ^{1/\alpha}.
\end{align*}
But by the result of Motzkin and Straus \cite{MoSt65}, we have
\[
2\sum_{\left\{  u,v\right\}  \in E\left(  G\right)  }\left\vert x_{u}%
\right\vert ^{\alpha}\left\vert x_{v}\right\vert ^{\alpha}\leq1-\frac{1}{r},
\]
and inequality (\ref{in1}) follows. Now inequality (\ref{in2}) follows from
(\ref{in1}) by Tur\'{a}n's theorem $2m<\left(  1-1/r\right)  n^{2}.$
\end{proof}

We shall need also the following proposition (Proposition 29, \cite{NikB})
whose proof we omit.

\begin{proposition}
\label{pro10}Let $\alpha\leq1,$ $k>1$ and $G_{1}$ and $G_{2}$ be graphs on the
same vertex set. If $G_{1}$ and $G_{2}$ differ in at most $k$ edges,then%
\[
\left\vert \lambda^{\left(  \alpha\right)  }\left(  G_{1}\right)
-\lambda^{\left(  \alpha\right)  }\left(  G_{2}\right)  \right\vert
\leq\left(  2k\right)  ^{1-1/\alpha}.
\]

\end{proposition}

Here is the main theorem about $\lambda^{\left(  \alpha\right)  }\left(
\mathcal{P}\right)  .$

\begin{theorem}
\label{thlam} Let $\alpha>1$ and let $\mathcal{F}$ be a family of graphs. If
the property $\mathcal{P}=Her\left(  \mathcal{F}\right)  $ is infinite, then
\[
\lambda^{\left(  \alpha\right)  }\left(  \mathcal{P}\right)  =\left\{
\begin{array}
[c]{l}%
1,\text{ if }\underline{\omega}\left(  \mathcal{F}\right)  =0;\\
1-\frac{1}{\beta\left(  \mathcal{F}\right)  -1},\text{ otherwise.}%
\end{array}
\right.  .
\]

\end{theorem}

\begin{proof}
First note the inequality
\[
\lambda^{\left(  \alpha\right)  }\left(  G\right)  \geq2e\left(  G\right)
/n^{2/\alpha},
\]
which follows by taking $\left(  x_{1},\ldots,x_{n}\right)  =\left(
n^{-1/\alpha},\ldots,n^{-1/\alpha}\right)  $ in (\ref{maxlam}). So we see
that
\[
\lambda^{\left(  \alpha\right)  }\left(  \mathcal{P}\right)  \geq\pi\left(
\mathcal{P}\right)  ,
\]
and this, together with Theorem \ref{thpi} gives $\lambda^{\left(
\alpha\right)  }\left(  \mathcal{P}\right)  =1$ if $\underline{\omega}\left(
\mathcal{F}\right)  =0$ and
\[
\lambda^{\left(  \alpha\right)  }\left(  \mathcal{P}\right)  \geq1-\frac
{1}{\beta\left(  \mathcal{F}\right)  -1}%
\]
otherwise. To finish the proof we shall prove that%
\[
\lambda^{\left(  \alpha\right)  }\left(  \mathcal{P}\right)  \leq1-\frac
{1}{\beta\left(  \mathcal{F}\right)  -1}%
\]

For the purposes of this proof, write $k_{r}\left(  G\right)  $ for the number
of $r$-cliques of $G$. Let $F\in\mathcal{F}$ be a complete $\beta$-partite
graph, which exists by the definition of $\beta\left(  \mathcal{F}\right)  ,$
and let $s$ be the maximum of the sizes of its vertex classes.

We recall the following particular version of the Removal Lemma, one of the
important consequences of the Szemer\'{e}di Regularity Lemma (\cite{Sze76}%
,\cite{Bol98}): \medskip

\textbf{Removal Lemma} \emph{Let }$r\geq2$\emph{ and }$\varepsilon>0.$\emph{
There exists }$\delta=\delta\left(  r,\varepsilon\right)  >0$\emph{ such that
if }$G$\emph{ is a graph of order }$n,$\emph{ with }$k_{r}\left(  G\right)
<\delta n^{r},$\emph{ then there is a graph }$G_{0}\subset G$\emph{ such that
}$e\left(  G_{0}\right)  \geq e\left(  G\right)  -\varepsilon n^{2}$\emph{ and
}$k_{r}\left(  G_{0}\right)  =0.\medskip$

In \cite{Nik08} we have proved the following theorem:\medskip

\textbf{Theorem A }\emph{For all }$r\geq2,$\emph{ and }$\varepsilon>0$\emph{
there exists }$\delta=\delta\left(  r,\varepsilon\right)  >0$\emph{ such that
if }$G$\emph{ a graph of order }$n$\emph{ with }$k_{r}\left(  G\right)
>\varepsilon n^{r},$\emph{ then }$G$\emph{ contains a }$K_{r}\left(  s\right)
$\emph{ with }$s=\left\lfloor \delta\log n\right\rfloor .\medskip$

Now let $\varepsilon>0,$ choose $\delta=\delta\left(  \beta,\varepsilon
\right)  $ as in the Removal Lemma, and set $t=r\left(  K_{r},K_{s}\right)  ,$
where $r\left(  K_{r},K_{s}\right)  $ is the Ramsey number of $K_{r}$ vs.
$K_{s}.$ If $G\in\mathcal{P}_{n},$ then $K_{\beta}\left(  t\right)  \nsubseteq
G$ as otherwise we see as in proof of Theorem \ref{thpi} that $G$ contains an
induced copy of $F.$ So by Theorem A, if $n$ is large enough, then $k_{\beta
}\left(  G\right)  \leq\delta n^{r}.$ Now by the Removal Lemma there is a
graph $G_{0}\subset G$ such that $e\left(  G_{0}\right)  \geq e\left(
G\right)  -\varepsilon n^{2}$ and $k_{\beta}\left(  G_{0}\right)  =0.$

By Propositions \ref{pro10} and \ref{maxlK}, for $n$ sufficiently large, we
see that
\[
\lambda^{\left(  \alpha\right)  }\left(  G\right)  \leq\lambda^{\left(
\alpha\right)  }\left(  G_{0}\right)  +\left(  2\varepsilon n\right)
^{2-2/\alpha}\leq\left(  1-\frac{1}{\beta-1}\right)  n^{2-2/\alpha}+\left(
2\varepsilon n\right)  ^{2-2/\alpha},
\]
and hence,
\[
\lambda^{\left(  \alpha\right)  }\left(  \mathcal{P},n\right)  n^{2/\alpha
-2}\leq1-\frac{1}{\beta-1}+\left(  2\varepsilon\right)  ^{2-2/\alpha}%
\]
Since $\varepsilon$ can be made arbitrarily small, we see that
\[
\lambda^{\left(  \alpha\right)  }\left(  \mathcal{P}\right)  \leq1-\frac
{1}{\beta-1},
\]
completing the proof of Theorem \ref{thlam}.
\end{proof}

\bigskip

To complete the picture, we need \ to determine the dependence of
$\lambda^{\left(  1\right)  }\left(  \mathcal{P}\right)  $ on $\mathcal{P}$.
Using the the well-known idea of Motzkin and Straus, we come up with the
following theorem, whose proof we omit

\begin{theorem}
$\lambda^{\left(  1\right)  }\left(  \mathcal{P}\right)  $Let $\mathcal{P}$ be
an infinite hereditary property. Then $\lambda^{\left(  1\right)  }\left(
\mathcal{P}\right)  =1$ if $\mathcal{P}$ contains arbitrary large cliques, or
$\lambda^{\left(  1\right)  }\left(  \mathcal{P}\right)  =1-1/r,$ where $r$ is
the size of the largest clique in $\mathcal{P}$.
\end{theorem}

\bigskip

\section{Concluding remarks}

In a cycle of papers the author has shown that many classical exremal results
like the Erd\H{o}s-Stone-Bolloabs theorem \cite{BoEr73}, the Stability Theorem
of Erd\H{o}s \cite{Erd66,Erd68} and Simonovits \cite{Sim68}, and various
saturation problems can be strengthened by recasting them for the largest
eigenvalue instead of the number of edges; see \cite{Nik11} for overview and references.

The results in the present note and in \cite{NikB} show that some of these
results can be extended further for $\lambda^{\left(  \alpha\right)  }\left(
G\right)  $ and $\alpha\geq1.$ A natural challenge here is to reprove
systematically all of the above problems by substituting $\lambda^{\left(
\alpha\right)  }\left(  G\right)  $ for the number of edges.\bigskip

\textbf{Acknowledgement}\ Thanks are due to Bela Bollob\'{a}s for useful
discussions.

\end{document}